\documentclass[12pt]{article}

\usepackage[utf8]{inputenc}
\usepackage[T1]{fontenc}
\usepackage{amsmath, amsfonts, amssymb, mathrsfs, amscd}
\usepackage{mathtools}
\usepackage{amsthm}
\usepackage[margin=2.5cm]{geometry}
\usepackage[unicode, colorlinks, linkcolor=blue, citecolor=blue, urlcolor=blue]{hyperref}
\usepackage{enumerate}
\usepackage{bm}

\numberwithin{equation}{section}

\theoremstyle{plain}
\newtheorem{theorem}{Theorem}
\newtheorem{lemma}{Lemma}

\newtheorem{corollary}{Corollary}
\newtheorem{thmx}{Theorem}

\theoremstyle{definition}

\newtheorem{remark}{Remark}

\title{%
  Estimate of the Rate of Convergence of Fourier Sums \\[0.3em]
  for Functions from Lebesgue Classes on a Set of Full Measure%
  \thanks{%
    This is a preprint of a manuscript submitted to \emph{Izvestiya: Mathematics}.
  }%
}

\author{%
  D.\,I.~Masyutin%
  \thanks{%
    N.N. Krasovskii Institute of Mathematics and Mechanics,
    Ural Branch of the Russian Academy of Sciences,
    Yekaterinburg, Russia.
    \href{mailto:newselin@mail.ru}{newselin@mail.ru}
  }%
}

\date{}

\begin{document}

\maketitle

\begin{abstract}
We obtain an estimate of the rate of convergence on a set of full measure
of partial sums of trigonometric Fourier series of functions from Lebesgue
classes and construct a counterexample showing the order sharpness of this
estimate. We derive a condition for Prinsheim convergence almost everywhere
of two-dimensional trigonometric Fourier series of functions from Lebesgue
classes in terms of the modulus of continuity.


\end{abstract}

\section{Introduction}

Let $\mathbb{N}, \mathbb{Z}, \mathbb{Z}_+, \mathbb{R}$ be the sets of natural,
integer, non-negative integer, and real numbers, respectively,
$p \in [1, +\infty), \; d\in\{1, 2\}$,
and let $\mathbb{T}^d := [-\pi, \pi)^d$ be the $d$-dimensional torus.
We denote by $L^p(\mathbb{T}^d)$ the linear normed space of Lebesgue-measurable
functions $f: \mathbb{T}^d \rightarrow \mathbb{R}$ with norm
\begin{equation*}
\|f\|_p := \left(\int\limits_{\mathbb{T}^d} |f(x)|^p \, dx \right)^{1/p}.
\end{equation*}

We introduce the modulus of continuity of a function $f$ in $L^p(\mathbb{T}^d)$:
\begin{equation*}
\omega(f, \delta)_p := \sup\limits_{h\in\mathbb{R}^d: |h| \leqslant \delta}
\| f(\cdot+h) - f \|_p,
\end{equation*}
where $|h|$ is the usual Euclidean norm in $\mathbb{R}^d$.

Let $d=1$, $n \in \mathbb{Z}_+$, and let $S_{n}(f, x)$ be the $n$-th partial sum
of the trigonometric Fourier series of $f$ at the point $x \in \mathbb{T}$.
For a Lebesgue-measurable subset $E \subset \mathbb{T}$ we denote by
$m(E)$ its one-dimensional Lebesgue measure.

For continuous $2\pi$-periodic functions, A.~Lebesgue~\cite{Lebesgue}
established the following uniform inequality
(see, e.g.,~\cite[Ch.~4, \S~3]{Bari}):
\begin{equation}\label{Lebesgue_est}
|f(x) - S_n(f, x)| \leqslant C E_n(f) \ln{(n+3)},
\end{equation}
where $E_n(f)$ is the best approximation of $f$ by trigonometric polynomials
of degree at most $n$ in the space of continuous functions,
and $C$ is an absolute constant.

In 1974, K.\,I.~Oskolkov~\cite[Theorem~1]{Oskolkov1} showed that if the
deviation of a function from its Fourier sums is estimated not at every point
$x \in \mathbb{T}$, but almost everywhere, then Lebesgue's estimate
\eqref{Lebesgue_est} can be essentially improved; namely, for all
$n \in \mathbb{Z}_+$ the following estimate holds:
\begin{equation}\label{Oskolkov_eq}
|f(x) - S_n(f, x)| \leqslant C(x) E_n(f)
\ln{\ln{\left(\frac{8 E_0(f)}{E_n(f)}\right)}},
\end{equation}
where $C(x)$ is a non-negative almost everywhere finite function.
Moreover, there exists a constant $A > 0$ such that for all $y > 0$,
\begin{equation*}
m\{x \in \mathbb{T}: C(x) > y\} \leqslant Ae^{-y/A}.
\end{equation*}

In the present paper we obtain an analogue of estimate~\eqref{Oskolkov_eq}
for the rate of approximation by partial sums of trigonometric Fourier series
of functions from Lebesgue classes. Let $E_n(f)_p$ denote the best
approximation of $f$ by trigonometric polynomials of degree at most $n$
in the space $L^p(\mathbb{T})$.

\begin{theorem}\label{theorem1}
Let $\varphi: [0, +\infty) \rightarrow [0, +\infty)$ be an increasing
function such that
\begin{equation}\label{condition}
\sum\limits_{k=1}^{\infty} \frac{1}{k\varphi(k)} < +\infty.
\end{equation}
Let $1 < p < \infty$ and let $f \in L^p(\mathbb{T})$ be not a trigonometric
polynomial. Then for every $n \in \mathbb{Z}_+$ the estimate
\begin{equation}\label{estimate}
|f(x) - S_n(f, x)| \leqslant C(x) E_n(f)_p
\varphi^{1/p}\left(\frac{E_0(f)_p}{E_n(f)_p}\right)
\end{equation}
holds, where $C(x) \in L^p(\mathbb{T})$ and
\begin{equation}\label{norm_C}
\|C\|_p^p \leqslant C_p \sum\limits_{k=1}^{\infty} \frac{1}{k\varphi(k)},
\end{equation}
with $C_p$ a positive constant depending only on $p$.
\end{theorem}

We note that in 1985 K.\,I.~Oskolkov~\cite[Theorem~2]{Oskolkov3} obtained,
in particular, an estimate for functions $f \in L(\mathbb{T})$. In our
notation it can be written as
\begin{equation*}
|f(x) - S_n(f, x)| \leqslant C(x) E_n(f)_1
\varphi \left(\frac{E_0(f)_1}{E_n(f)_1} \right) \ln{n},
\end{equation*}
where $\varphi: [0; +\infty) \rightarrow [0; +\infty)$ is an increasing
function satisfying~\eqref{condition}, and $C(x)$ is almost everywhere finite.
Moreover, there exists a constant $A>0$ such that for all $y>0$,
\begin{equation*}
m\{x \in \mathbb{T}: C(x) > y\} \leqslant \frac{A}{y}.
\end{equation*}

As is well known, Jackson's inequality in $L^p(\mathbb{T})$ holds
(see~\cite{Quade} and also~\cite[Addenda, \S~7]{Bari}): for any $p \geqslant 1$
there exists a constant $C_p>0$ such that for any $f \in L^p(\mathbb{T})$
and $n \in \mathbb{N}$,
\begin{equation*}
E_n(f)_p \leqslant C_p \omega\left(f, \frac{1}{n}\right)_p.
\end{equation*}
From Theorem~\ref{theorem1} and Jackson's inequality we obtain the following.

\begin{corollary}\label{corollary}
Let $1 < p < \infty$, $f \in L^p(\mathbb{T})$, and let
$\varphi: [0, +\infty) \rightarrow [0, +\infty)$ be an increasing function
satisfying condition~\eqref{condition}. Suppose that in some neighbourhood
of zero the function $x \varphi^{1/p}(1/x)$ is increasing. Then there exists
$n_0 \in \mathbb{N}$ such that for $n \geqslant n_0$,
\begin{equation*}
|f(x) - S_n(f, x)| \leqslant C(x)
\omega\left(f, \frac{1}{n}\right)_p
\varphi^{1/p}\left(\frac{1}{\omega\left(f, \frac{1}{n}\right)_p}\right),
\end{equation*}
where $C(x) \in L^p(\mathbb{T})$ and its norm satisfies~\eqref{norm_C}.
\end{corollary}

\begin{remark}
If $x \varphi^{1/p}(1/x)$ is increasing on $(0, A]$, then in Corollary~\ref{corollary}
one can take $n_0$ such that $C_p \omega\left(f, 1/n_0 \right)_p \leqslant A$.
\end{remark}

It should be noted that a function from $L^p(\mathbb{T})$ with sufficiently
rapidly decreasing best approximations coincides almost everywhere with some
continuous function on $\mathbb{T}$ (see~\cite[Theorem~5 and Theorem~F]{Ulyanov}).
Therefore, from Lebesgue's estimate for continuous functions~\eqref{Lebesgue_est}
and the Konyushkov--Stechkin inequality
(see~\cite[Theorem~2]{KonStech})
\begin{equation*}
E_n(f) \leqslant A_p \left[E_n(f)_p(n+1)^{1/p} +
\sum\limits_{k=n+1}^{\infty} k^{1/p-1} E_k(f)_p \right]
\end{equation*}
we obtain the uniform Lebesgue inequality for functions in $L^p(\mathbb{T})$:
for almost all $x \in \mathbb{T}$,
\begin{equation}\label{Lebesque_Lp}
|f(x) - S_n(f, x)| \leqslant C A_p
\left[E_n(f)_p(n+1)^{1/p} +
\sum\limits_{k=n+1}^{\infty} k^{1/p-1} E_k(f)_p \right] \ln{n}.
\end{equation}
For example, if $E_n(f)_p = e^{-n^{\alpha}}, \alpha > 0$, then estimate
\eqref{estimate} does not give essential improvements over
\eqref{Lebesque_Lp}. However, if $E_n(f)_p$ decreases not so rapidly,
then estimate~\eqref{estimate} is more precise for almost all $x \in \mathbb{T}$,
and in this case~\eqref{estimate} is sharp, as stated in the following theorem.

\begin{theorem}\label{theorem2}
Let $\varphi: [0, +\infty) \rightarrow [0, +\infty)$ be an increasing function
such that condition~\eqref{condition} fails and
$\varphi(x) = o(x^{\alpha})$ as $x \to \infty$ for every $\alpha>0$. Let $\{ \varepsilon_n \}_{n=0}^{\infty}$ be a sequence of positive numbers
monotonically tending to zero such that for some positive $q$ the sequence
$\{n^q \varepsilon_n \}_{n=0}^{\infty}$ is non-decreasing. Then for every $1 < p < \infty$ there exists $F \in L^p(\mathbb{T})$ such that
\begin{equation}\label{E_n(f)_cond}
E_n(F)_p \leqslant \varepsilon_n,
\end{equation}
and for almost all $x \in \mathbb{T}$,
\begin{equation}\label{limsup}
\varlimsup\limits_{n \to \infty}
\frac{|F(x) - S_n(F, x)|}{\varepsilon_n
\varphi^{1/p}(\varepsilon_0/\varepsilon_n)} = +\infty.
\end{equation}
\end{theorem}

It is worth noting that in the case of continuous $2\pi$-periodic functions,
a similar result was obtained by K.\,I.~Oskolkov~\cite[Theorem~2]{Oskolkov1},
asserting that in the case of sufficiently slowly decreasing best approximations,
estimate~\eqref{Oskolkov_eq} is also sharp.

We now turn to the case $d = 2$. Let $S_{n_1, n_2}(f, x_1, x_2)$ be the
$(n_1, n_2)$-th rectangular partial sum of the double trigonometric Fourier
series of $f$ at the point $(x_1, x_2) \in \mathbb{T}^2$. The Fourier series
of $f$ is said to converge in the sense of Prinsheim (or over rectangles)
at a point $(x_1, x_2) \in \mathbb{T}^2$ if the limit
\begin{equation*}
\lim\limits_{\min{\{n_1, n_2\}} \to \infty}
S_{n_1, n_2}(f, x_1, x_2)
\end{equation*}
exists.

In 1971, C.~Fefferman~\cite{Fefferman} constructed a continuous function
$F$ on $\mathbb{T}^2$ whose Fourier series diverges in the Prinsheim sense
almost everywhere on $\mathbb{T}^2$. Consequently, in the two-dimensional case
there is no sufficient condition for Prinsheim convergence almost everywhere
in terms of membership of a function in some space $L^p(\mathbb{T}^2)$.
Later, M.~Bakhbukh and E.\,M.~Nikishin~\cite{BahNik} proved that there exists
a continuous function $f$ on $\mathbb{T}^2$ with Fourier series diverging
on a set of positive measure and with modulus of continuity satisfying
$\omega(f, \delta) = O(\ln^{-1}{(1/\delta)})$ as $\delta \to +0$
(see also~\cite{Bak97} for an extension of this result to the even-dimensional
case). In~\cite[Theorem~3]{Oskolkov1}, K.\,I.~Oskolkov proved that if the
modulus of continuity of a continuous function $f$ on $\mathbb{T}^2$ satisfies
\begin{equation*}
\omega(f, \delta) = o\left( \frac{1}{\ln(1/\delta)
\ln{\ln{\ln(1/\delta)}}} \right), \quad \delta \rightarrow +0,
\end{equation*}
then the Fourier series of $f$ converges in the Prinsheim sense almost everywhere.

An analogue of this result for the case $L^p,\; p=1, 2$, follows from the
following result of L.\,V.~Zhizhiashvili (1975)~\cite{Zhizhiashvili}
(see also~\cite[Ch.~1, \S~1, 2]{Dya92}). We formulate it for $d=2$.
Let $\alpha = (\alpha_1, \alpha_2)$ consist of zeros and ones,
$p \in \{1, 2\}$, $f \in L^p(\mathbb{T}^2)$, $s \in \mathbb{T}^2$. Set
\begin{gather*}
\Delta_{s, (1, 1), 1}(f, x) = \Delta_{s, (1, 0), 1}(f, x)
:= f(x_1 + s_1,  x_2) - f(x_1 - s_1, x_2), \\
\Delta_{s, (1, 1), 2}(f, x) = \Delta_{s, (0, 1), 2}(f, x)
:= f(x_1,  x_2 + s_2) - f(x_1, x_2 - s_2),\\
\Delta_{s, (0, 1), 1}(f, x) = \Delta_{s, (1, 0), 2}(f, x) := f(x).
\end{gather*}
Further, let
\begin{equation*}
\Delta_{s, \alpha}(f, x) := \Delta_{s, \alpha, 1}(f, x) \circ
\Delta_{s, \alpha, 2}(f, x).
\end{equation*}
Then, if
\begin{equation}\label{Zhi_cond}
\sum\limits_{\alpha \neq (0, 0)}
\int\limits_{\mathbb{T}^2}
\int\limits_{\mathbb{T}^{\alpha_1 + \alpha_2}}
\frac{|\Delta_{s, \alpha}(f, x)|^p}{|s_1^{\alpha_1} s_2^{\alpha_2}|}
\, (ds_1)^{\alpha_1} \, (ds_2)^{\alpha_2} \, dx < \infty,
\end{equation}
then the Fourier series of $f$ converges in the Prinsheim sense almost everywhere.
In 1994, A.\,M.~Stokolos proved (see~\cite[Lemma~3]{Stokolos}) that, in the
particular case $p=1$, the inequality
\begin{equation}\label{Stok_cond}
\int\limits_{\mathbb{T}^2}
\int\limits_{\mathbb{T}^{\alpha_1 + \alpha_2}}
\frac{|\Delta_{s, (1, 1)}(f, x)|^p}{|s_1 s_2|}
\, ds_1 \, ds_2 \, dx \leqslant C_p
\int\limits_0^{\pi} \frac{\omega(f, t)_p^p}{t}
\ln{\frac{2\pi}{t}} \, dt
\end{equation}
holds for some constant $C_p>0$. Moreover, Stokolos's proof of
inequality~\eqref{Stok_cond} is valid for $p \geqslant 1$. Consequently, the
condition
\begin{equation}\label{Stok_cond_2}
\sum\limits_{n=1}^{\infty}
\frac{\omega\left(f, \frac{1}{n} \right)_p^p}{n} \ln{n} < \infty
\end{equation}
is sufficient for condition~\eqref{Zhi_cond} to hold. Therefore, if
$f \in L^p(\mathbb{T}^2)$ satisfies condition~\eqref{Stok_cond_2}
for $p=1, 2$, then the Fourier series of $f$ converges in the Prinsheim sense
almost everywhere. Note that an analogous criterion for $p > 2$ does not
hold, by virtue of the above-cited work of M.~Bakhbukh and E.\,M.~Nikishin.
We also mention that in~\cite[Ch.~1, \S~2]{Dya92} the following question
is posed: for $1 < p < 2$, is condition~\eqref{Zhi_cond} sufficient for
Prinsheim convergence almost everywhere of the Fourier series of
$f \in L^p(\mathbb{T}^2)$? In the final part of this paper we show that
for condition~\eqref{Stok_cond_2} the answer to an analogous question is
positive. Namely, we extend the condition in terms of the modulus of
continuity sufficient for the Fourier series to converge in the Prinsheim
sense almost everywhere from the case $p=2$ to the case $1 < p < \infty$.

\begin{theorem}\label{theorem3}
Let $1 < p < \infty$ and $f \in L^p(\mathbb{T}^2)$. If
\begin{equation}\label{cond_t_3}
\sum\limits_{n=1}^{\infty}
\frac{\omega\left(f, \frac{1}{n}\right)_p^p}{n}
(\ln{n})^{p-1} < \infty,
\end{equation}
then the Fourier series of $f$ converges in the Prinsheim sense almost everywhere.
\end{theorem}

Note that in the case $1 < p < 2$, condition~\eqref{cond_t_3} is more precise
than condition~\eqref{Stok_cond_2}. From Theorem~\ref{theorem3} it follows
that if for some function $\varphi:[0, +\infty) \rightarrow [0, +\infty)$
satisfying condition~\eqref{condition} and for $f \in L^p(\mathbb{T}^2)$ the
condition
\begin{equation*}
\omega(f, \delta)_p = O\left(\frac{1}{\ln{(1/\delta)}
\varphi^{1/p}(\ln{(1/\delta)})}\right)
\end{equation*}
holds as $\delta \to +0$, then the Fourier series of $f$ converges in the
Prinsheim sense almost everywhere. For the higher-dimensional case $d > 2$,
similar results were recently obtained by Goginava~\cite{Goginava2026} (see also~\cite{Goginava2026b}).

\section{Auxiliary results}

Let $S^*(f, x) = \sup\limits_{n \in \mathbb{Z}_+} |S_n(f, x)|$ be the
majorant of the partial sums of the Fourier series of $f \in L^p(\mathbb{T})$.
L.~Carleson~\cite{Carleson} proved that if $f \in L^2(\mathbb{T})$, then
the trigonometric Fourier series of $f$ converges almost everywhere.
R.~Hunt~\cite{Hunt} extended this almost everywhere convergence result to
functions in $L^p(\mathbb{T}),\; p > 1$. Moreover, he obtained the following
remarkable estimate.

\begin{thmx}\label{Hunt}
If $1 < p < \infty$, then there exists a constant $\overline{C}_p > 0$,
depending only on $p$, such that for all $f \in L^p(\mathbb{T})$,
\begin{equation}\label{eq_Hunt}
\|S^*(f, \cdot)\|_p^p \leqslant \overline{C}_p \| f \|_p^p.
\end{equation}
From~\eqref{eq_Hunt} it follows that the Fourier series of any
$f \in L^p(\mathbb{T})$ converges almost everywhere.
\end{thmx}

Let $E$ be a measurable subset of $\mathbb{T}$. Denote
\begin{equation*}
E^y := \{x + y: x \in E \}
\end{equation*}
the translate of $E$ by $y$.

\begin{thmx}[{\cite[Ch.~13, \S~1, Lemma~1.24]{Zigmund2}}]\label{Borel-Kantelli}
Let $\{ E_n \}_{n=1}^{\infty}$ be a sequence of measurable subsets of
$\mathbb{T}$ such that
\begin{equation*}
\sum\limits_{n=1}^{\infty} mE_n = \infty.
\end{equation*}
Then there exists a sequence $\{ x_n \}_{n=1}^{\infty} \subset \mathbb{T}$
such that
\begin{equation*}
m\left(\varlimsup\limits_{n \to \infty} E_n^{x_n} \right)
:= m\left(\bigcap\limits_{k=1}^{\infty}
\bigcup\limits_{n=k}^{\infty} E_n^{x_n} \right) = 2 \pi.  
\end{equation*}
\end{thmx}

Let $\mathcal{M}$ denote the Hardy--Littlewood maximal function
(see~\cite{HL1}), defined for $f \in L(\mathbb{T})$ by
\begin{equation*}
\mathcal{M}f(x) := \sup\limits_{0 < h \leqslant \pi}
\frac{1}{2h} \int\limits_{x-h}^{x+h} |f(t)| \, dt.
\end{equation*}
This operator satisfies the following strong $(p, p)$ inequality
(see, e.g.,~\cite[Ch.~1, \S~13]{Zigmund1}).

\begin{thmx}\label{maximal}
Let $1 < p < \infty$ and $f \in L^p(\mathbb{T})$. Then there exists a
constant $A_p > 0$, depending only on $p$, such that
\begin{equation*}
\|\mathcal{M}f\|_p \leqslant A_p \|f\|_p.
\end{equation*}
\end{thmx}

In 1931, Hardy and Littlewood~\cite{HL2} proved the following classical
theorem (see, e.g.,~\cite[Ch.~12, \S~6]{Zigmund2}).

\begin{thmx}\label{Hardy-Littlewood}
Let $\{a_n\}_{n=1}^{\infty}$ be a monotone sequence of non-negative numbers
tending to zero, and let $p \in (1, +\infty)$. Then the function
$f(x) = \sum\limits_{n=1}^{\infty} a_n \cos{nx}$ belongs to $L^p(\mathbb{T})$
if and only if
\begin{equation*}
\sum\limits_{n=1}^{\infty} a_n^p n^{p-2} < +\infty.
\end{equation*}
Moreover, we have the inequalities
\begin{equation}\label{HL}
C_1 \sum\limits_{n=1}^{\infty} a_n^p n^{p-2} \leqslant \|f\|_p^p
\leqslant C_2 \sum\limits_{n=1}^{\infty} a_n^p n^{p-2},
\end{equation}
where $C_1, C_2$ are positive constants depending only on $p$.
\end{thmx}

Let $f \in L^p(\mathbb{T}^2)$ and fix $t \in \mathbb{T}$. Denote
\begin{gather*}
I_p(f_t, h) := \left(\int\limits_{\mathbb{T}}
|f(t, x+h) - f(t, x)|^p \, dx \right)^{1/p}, \\
\omega(f_t, \delta)_p := \sup\limits_{h\in\mathbb{R}: |h| \leqslant \delta}
I_p(f_t, h).
\end{gather*}
Note that $\omega(f_t, \delta)_p$ is the modulus of continuity of the function
$f_t(\cdot) := f(t, \cdot) \in L^p(\mathbb{T})$ for almost all $t \in \mathbb{T}$.
In the proof of Theorem~\ref{theorem3} we shall need the following auxiliary
statement.

\begin{lemma}\label{lemma}
Let $p \in [1, +\infty), f\in L^p(\mathbb{T}^2)$. Then for all $\delta>0$,
\begin{equation}\label{lemma_eq}
\int\limits_{\mathbb{T}} \omega(f_t, \delta)_p^p \, dt \leqslant C_p
\omega(f, \delta)_p^p,
\end{equation}
where $C_p$ is a positive constant depending only on $p$.
\end{lemma}

\begin{proof}
It is known that for $\delta > 0$ (see, e.g.,~\cite[Lemma~7.2]{PP}),
\begin{equation}\label{PP}
\omega(f_t, \delta)_p \leqslant \frac{c}{\delta}
\int\limits_0^{\delta} I_p(f_t, h) \, dh,
\end{equation}
where $c$ is an absolute positive constant. Then, using~\eqref{PP}, we get
\begin{multline*}
\int\limits_{\mathbb{T}} \omega(f_t, \delta)_p^p \, dt
\leqslant \left( \frac{c}{\delta} \right)^p
\int\limits_{\mathbb{T}} \left( \int\limits_0^{\delta}
I_p(f_t, h) \, dh \right)^p \, dt
\leqslant \frac{c^p}{\delta}
\int\limits_{\mathbb{T}} \int\limits_0^{\delta}
I_p(f_t, h)^p \, dh \, dt \\
= \frac{c^p}{\delta}
\int\limits_0^{\delta}
\int\limits_{\mathbb{T}} I_p(f_t, h)^p \, dt \, dh
\leqslant \frac{c^p}{\delta}
\int\limits_0^{\delta} \omega(f, h)_p^p \, dh
\leqslant c^p \omega(f, \delta)_p^p.
\end{multline*}
The lemma is proved.
\end{proof}

\section{Proof of Theorem~\ref{theorem1}}

We shall apply the key ideas from the proof of Theorem~1 in the continuous
case considered in~\cite{Oskolkov1}, using the specifics of the $L^p$ case
instead of $C$.

Fix $n \in \mathbb{Z}_+$. Consider the quantity
\begin{equation}\label{M_n(f)}
M_n(f, \varphi, x) :=
\max\limits_{0 \leqslant m \leqslant n}
\frac{|f(x) - S_m(f, x)|}{
E_m(f)_p \varphi^{1/p}(E_0(f)_p/E_m(f)_p)}.
\end{equation}
Clearly, in the conclusion of the theorem one can take
\begin{equation*}
C(x) := \sup\limits_{n \in \mathbb{Z}_+} M_n(f, \varphi, x)
= \lim\limits_{n \to \infty} M_n(f, \varphi, x).
\end{equation*}
We estimate the norm $\|M_n(f, \varphi, \cdot)\|_p$. To this end, define
\begin{equation}\label{def_n_k}
n_0 := -1,\quad n_k := \max\{n \in \mathbb{Z}_+:
E_n(f)_p \geqslant 2^{-k} E_0(f)_p\},\quad k \in \mathbb{N}.
\end{equation}
Since $E_n(f)_p$ is non-increasing, the sequence $\{n_k\}_{k=1}^{\infty}$
is non-decreasing.

Fix $k \in \mathbb{Z}_+$ such that $n_{k+1} > n_{k}$, let $T_k$ be a
trigonometric polynomial of best approximation of order $n_k + 1$ for $f$
in $L^p(\mathbb{T})$, and set $g_k(x) := f(x) - T_k(x)$. Then by the
definition~\eqref{def_n_k} of $\{n_k\}_{k=1}^{\infty}$ we have
\begin{gather*}
\| g_k \|_p = E_{n_k+1}(f)_p \leqslant 2^{-k} E_0(f)_p,\quad k \in \mathbb{Z}_+. \\
2^{-(k+1)} E_0(f)_p \leqslant E_n(f)_p \leqslant 2^{-k} E_0(f)_p,\quad
n = n_k+1, \dots,  n_{k+1}.
\end{gather*}
From these inequalities, for any $n \in \{n_k+1, \dots,  n_{k+1} \}$,
\begin{equation}\label{ineq_t_1}
E_n(f)_p \varphi^{1/p}\left(\frac{E_0(f)_p}{E_n(f)_p}\right)
\geqslant 2^{-(k+1)} E_0(f)_p \varphi^{1/p}(2^k)
\geqslant 2^{-1} \|g_k\|_p \varphi^{1/p}(2^k).
\end{equation}

By Theorem~\ref{Hunt}, for every $k \in \mathbb{Z}_+$ we have
$\lim\limits_{n \to \infty} S_n(g_k, x) = g_k(x)$ almost everywhere;
hence, on a set of full measure,
\begin{equation}\label{majorant}
g_k(x) \leqslant S^*(g_k, x).
\end{equation}
Without loss of generality, we may consider points only in the intersection
over all $k \in \mathbb{Z}_+$ of such sets. For any point $x$ in this
intersection, there exists a number $m(x) \in \{0, ..., n\}$ at which the
maximum in~\eqref{M_n(f)} is attained, and choose $k$ so that $m(x) \in \{n_k+1, \dots, n_{k+1}\}$. In particular, for this
$k$ we have $n_{k+1} > n_k$.

For our $n$, define $\overline{k} := \min{\{k:n \leqslant n_{k+1}\}}$.
Then, using the identity
\begin{equation*}
f(x) - S_{m(x)}(f, x) = g_k(x) - S_{m(x)}(g_k, x),
\end{equation*}
the definition~\eqref{M_n(f)} and estimates~\eqref{ineq_t_1},~\eqref{majorant},
we obtain
\begin{multline*}
\|M_n(f, \varphi, \cdot)\|_p^p =
\int\limits_{\mathbb{T}} |M_n(f, \varphi, x)|^p \, dx \\
\leqslant 2^p \sum\limits_{k=0}^{\overline{k}}
\int\limits_{\{x: m(x) \in \{n_k+1, \dots,  n_{k+1}\}\}}
\left(\frac{|g_k(x) - S_{m(x)}(g_k, x)|}{
\|g_k\|_p \varphi^{1/p}(2^k)}\right)^p \, dx \\
\leqslant 4^p \sum\limits_{k=0}^{\overline{k}}
\int\limits_{\mathbb{T}}
\left(S^*\left(\frac{g_k}{\|g_k\|_p \varphi^{1/p}(2^k)}, x \right)\right)^p \, dx.
\end{multline*}
From inequality~\eqref{eq_Hunt} it follows that
\begin{equation*}
\|M_n(f, \varphi, \cdot)\|_p^p
\leqslant 4^p \overline{C}_p
\sum\limits_{k=0}^{\overline{k}} \frac{1}{\varphi(2^k)}.
\end{equation*}
Hence, using monotone convergence theorem, we derive
\begin{multline*}
\|C\|_p^p =
\left\|\sup\limits_{n \in \mathbb{Z}_+}
\frac{|f(\cdot) - S_n(f, \cdot)|}{
E_n(f)_p \varphi^{1/p}(E_0(f)_p/E_n(f)_p)}\right\|_p^p \\
= \left\|\lim\limits_{n \to \infty} M_n(f, \varphi, \cdot)\right\|_p^p
\leqslant \lim\limits_{k \to \infty} 4^p \overline{C}_p
\sum\limits_{k=0}^{k} \frac{1}{\varphi(2^k)}
\leqslant 4^p \overline{C}_p
\sum\limits_{k=0}^{\infty} \frac{1}{\varphi(2^k)}.
\end{multline*}
By the hypothesis of the theorem, the series on the right converges.
Thus the required inequality is proved; in particular, $C(x) \in L^p(\mathbb{T})$.
Theorem~\ref{theorem1} is proved.

\section{Proof of Theorem~\ref{theorem2}}

Before giving the direct proof of the theorem, we describe the main ideas
of the construction. Following K.\,I.~Oskolkov's approach~\cite{Oskolkov1},
we construct the function $F$ as a series of trigonometric polynomials $Q_k$.
For convenience in estimating the best approximations of $F$, we shall
construct the sequence $n_k$, which will be the sequence of degrees of the
polynomials $Q_k$, so that it increases like a geometric progression.
The coefficients of the polynomials $Q_k$ are constructed so that the
$L^p$-norms of the polynomials are uniformly bounded and the deviations of
the partial Fourier sums of certain orders from $Q_k$ are bounded below by the
required amount on sets $\mathcal{H}_k$ of sufficiently large measure.
For the polynomials to satisfy the second requirement, we introduce a sequence
$m_k$ representing the factor $\varphi(\varepsilon_0/\varepsilon_{n_k})$.
By the hypothesis of Theorem~\ref{theorem2}, the sequence
$\varepsilon_{n_k}$ decreases like a geometric progression. Therefore, since
condition~\eqref{condition} fails for $\varphi$, we can use the translation
Theorem~\ref{Borel-Kantelli}, which allows us to arrange the polynomials $Q_k$
so that the union of the translated sets $\mathcal{H}_k$ covers the torus
almost everywhere.

We now begin the construction. For the increasing function $\varphi$, we
construct a non-decreasing function $\psi$ such that
\begin{equation}\label{frac}
\sum\limits_{n=1}^{\infty} \frac{1}{n \psi(n)} = +\infty,\quad
\frac{\psi(n)}{\varphi(n)} \xrightarrow[n \to \infty]{} +\infty.
\end{equation}

Fix $q' > q$ satisfying a condition to be formulated later. Then define the
sequence of natural numbers $\{n_k\}_{k=1}^{\infty}$ as follows: let
$n_1 = 1$, and for $k > 1$ put
\begin{equation*}
n_k := \min{\left\{n \in \mathbb{N}:
\varepsilon_n < 10^{-q'} \varepsilon_{n_{k-1}} \right\}}.
\end{equation*}
By the hypothesis that $\{n^{q'} \varepsilon_n\}_{n=0}^{\infty}$ is
non-decreasing, we obtain
\begin{equation*}
(10 n_k)^{q'} \varepsilon_{10 n_k} \geqslant n_k^{q'} \varepsilon_{n_k},
\end{equation*}
hence
\begin{equation}\label{n_k}
n_{k+1} > 10 n_k, \quad k \in \mathbb{N}.
\end{equation}
We also note that
\begin{equation}\label{eps_n_k}
10^{-q'} \varepsilon_{n_{k+1}} \leqslant \varepsilon_{n}
< 10^{-q'} \varepsilon_{n_k},\quad n \in \{n_{k+1}, ..., n_{k+2}-1 \}.
\end{equation}

Denote
\begin{equation}\label{m_k}
m_1 := 1,\quad m_k := \min{\left(\left\lceil
\psi\left(\frac{\varepsilon_0}{\varepsilon_1} 10^{k q'}\right) \right\rceil,
\left\lfloor \frac{n_k}{8} \right\rfloor \right)} + 1,\quad k \geqslant 2.
\end{equation}
From inequalities~\eqref{eps_n_k} and the condition on $\varphi$, it follows
that $\varphi(\varepsilon_0/\varepsilon_{3n_k}) = o(10^k)$ as $k \to \infty$.
From this and inequality~\eqref{n_k} it follows that
$\varphi(\varepsilon_0/\varepsilon_{3n_k}) = o(n_k)$ as $k \to \infty$.
Therefore, by property~\eqref{frac} and definition~\eqref{m_k},
\begin{equation}\label{m_k_2}
\frac{m_k}{\varphi(\varepsilon_0/\varepsilon_{3n_k})}
\xrightarrow[k \to \infty]{} \infty.
\end{equation}

Consider the polynomials
\begin{equation*}
Q_k(x) := 2 \sin{3n_kx}
\sum\limits_{n=m_k}^{2m_k} \frac{\cos{nx}}{n^{1 - 1/p}},\quad x \in \mathbb{T}.
\end{equation*}
Note that
\begin{equation*}
Q_k(x) =
\sum\limits_{n=m_k}^{2m_k}
\frac{\sin{(3n_k-n)x}}{n^{1 - 1/p}} +
\sum\limits_{n=m_k}^{2m_k}
\frac{\sin{(3n_k+n)x}}{n^{1 - 1/p}}
\end{equation*}
and $3n_k + 2m_k \leqslant 5 n_k$ and $3n_k - 2m_k \geqslant n_k$, whence
\begin{gather}
|Q_k(x) - S_l(Q_k, x)| = Q_k(x),\quad l < n_k,\label{m}\\
|Q_k(x) - S_l(Q_k, x)| = 0,\quad l > 5n_k - 1.\label{5m}
\end{gather}

Define the sets
\begin{multline}\label{A_k_def}
A_k := \left\{x \in \left(0, \frac{\pi}{12m_k}\right]:
|\sin{3n_kx}| \geqslant \frac{\sqrt{3}}{2}\right\} \\
= \left\{x \in \left(0, \frac{\pi}{12m_k}\right]:
|\cos{3n_kx}| \leqslant \frac{1}{2} \right\},\quad k \in \mathbb{N} \setminus \{1\}.
\end{multline}
The set $A_k$ consists of at least $\left\lfloor n_k/(4m_k) \right\rfloor
\geqslant 1$ pairwise disjoint intervals, each of length $\pi/(9 n_k)$.
Therefore,
\begin{equation}\label{mA_k}
mA_k \geqslant \frac{\pi}{72 m_k}.
\end{equation}
For $x \in A_k$, the following inequalities hold:
\begin{multline}\label{estimate2}
|Q_k(x) - S_{3n_k}(Q_k, x)|
= \left|\sum\limits_{n=m_k}^{2m_k}
\frac{\sin{(3n_k+n)x}}{n^{1 - 1/p}} \right| \\
\geqslant |\sin{3n_kx}|
\left| \sum\limits_{n=m_k}^{2m_k}
\frac{\cos{nx}}{n^{1 - 1/p}} \right|
- |\cos{3n_kx}|
\left| \sum\limits_{n=m_k}^{2m_k}
\frac{\sin{nx}}{n^{1 - 1/p}} \right| \\
\geqslant \frac{3m_k}{4(2m_k)^{1 - 1/p}}
- \frac{m_k}{4m_k^{1 - 1/p}} = c m_k^{1/p},
\end{multline}
where $c := 1/8 > 0$.

For natural numbers $k$ and $l$, consider the sets
\begin{gather}
F_k := \left\{ x \in \mathbb{T}:
|Q_k(x) - S_{3n_k}(Q_k, x)| \geqslant c m_k^{1/p} \right\}, \label{F_k}  \\
G_l := \left\{x \in \mathbb{T}:
\left| \sum\limits_{k=l+1}^{\infty}
\varepsilon_{n_k} Q_k(x - x_k) \right| >
\varepsilon_{n_l} \frac{c m_l^{1/p}}{2}\right\},\label{G_l}
\end{gather}
where $\{x_k\}_{k=1}^{\infty}$ is a sequence of real numbers to be defined later.
From inequality~\eqref{estimate2} and definitions~\eqref{A_k_def},~\eqref{F_k}
it follows that $A_k \subset F_k$. Then from estimate~\eqref{mA_k} there exists
a constant $B := \pi/72>0$ such that
\begin{equation}\label{mF_k}
mF_k \geqslant \frac{B}{m_k},\quad k \in \mathbb{N} \setminus \{1\}.
\end{equation}

We use Theorem~\ref{Hardy-Littlewood}, namely the right-hand inequality
in~\eqref{HL}. By this inequality, there exists a constant $C>0$ such that
for every natural number $k \geqslant 2$, the norm of the polynomials $Q_k$
is estimated as follows:
\begin{equation*}
\|Q_k\|_p^p \leqslant C
\sum\limits_{n=m_k}^{2m_k} \frac{1}{n}
\leqslant C \int\limits_{m_k-1}^{2m_k} \frac{1}{x} \, dx
\leqslant C \ln{3}.
\end{equation*}
From this estimate and inequalities~\eqref{eps_n_k}, for all $l \in \mathbb{N}$
we obtain
\begin{equation}\label{Lp}
\left\| \sum\limits_{k=l+1}^{\infty}
\varepsilon_{n_k} Q_k(\cdot-x_k) \right\|_p
\leqslant \sum\limits_{k=l+1}^{\infty}
\varepsilon_{n_k} \| Q_k \|_p \leqslant (C \ln{3})^{\frac{1}{p}}
\sum\limits_{k=1}^{\infty} 10^{-kq'} \varepsilon_{n_l}
= C(p, q') \varepsilon_{n_l},
\end{equation}
where $C(p, q') := \frac{(C \ln{3})^{1/p}}{10^{q'} - 1}$. Choose $q'$ so large
that
\begin{equation*}
C(p, q') \leqslant \frac{c(B/2)^{1/p}}{2}.
\end{equation*}
Then, using Chebyshev's inequality and~\eqref{Lp}, from definition~\eqref{G_l}
we derive
\begin{equation}\label{mG_l}
mG_l \leqslant \frac{B}{2m_l},\quad l \in \mathbb{N}.
\end{equation}

Set $\mathcal{H}_k = F_k \setminus G_k$. From inequalities~\eqref{mF_k},
~\eqref{mG_l}, the definition of $m_k$, and the condition on $\psi$
in~\eqref{frac}, it follows that
\begin{equation*}
\sum_{k=1}^{\infty} m\mathcal{H}_k = +\infty.
\end{equation*}
By Theorem~\ref{Borel-Kantelli}, there exists a sequence $\{ x_k \}_{k=1}^{\infty}$
such that the measure of the set
\begin{equation}\label{Hcal}
\mathcal{H} := \varlimsup\limits_{k \to \infty}\mathcal{H}_k^{x_k}
\end{equation}
is equal to $2\pi$. Define the function $f$ by
\begin{equation*}
f(x) := \sum\limits_{k=1}^{\infty}
\varepsilon_{n_{k}} Q_k(x-x_k).
\end{equation*}
From estimate~\eqref{Lp}, it follows that $f \in L^p(\mathbb{T})$.

Let $n \in \mathbb{N}$ and let $l$ be defined by $n \in \{n_l, ..., n_{l+1}-1\}$.
Then the degree of the polynomial
$\sum\limits_{k=1}^{l-1} \varepsilon_{n_k} Q_k(x-x_k)$ does not exceed
$3n_{l-1} + 2m_{l-1}$, which is bounded above by $10n_{l-1}$, and by
~\eqref{n_k} this does not exceed $n$. Hence, using~\eqref{Lp} and
~\eqref{eps_n_k}, we obtain that there exists a constant $C'(p, q') > 0$
such that
\begin{equation}\label{E_n}
E_n(f)_p \leqslant \left\|f -
\sum\limits_{k=1}^{l-1} \varepsilon_{n_k} Q_k(\cdot-x_k) \right\|_p
\leqslant C'(p, q') \varepsilon_n.
\end{equation}

By the definition~\eqref{Hcal} of the set $\mathcal{H}$, for every $x \in \mathcal{H}$
there exists an infinite sequence of natural numbers $\{ l_i \}_{i=1}^{\infty}$
such that $x \in \mathcal{H}_{l_i}^{x_{l_i}}$ for all natural numbers $i$.
But if $x \in \mathcal{H}_l^{x_l}$, then by~\eqref{m},~\eqref{5m}, and
~\eqref{n_k},
\begin{equation*}
f(x) - S_{3n_l}(f, x) =
\varepsilon_{n_l}\left(Q_l(x-x_l) - S_{3n_l}(Q_l(\cdot - x_l),x) \right)
+ \sum\limits_{k=l+1}^{\infty} \varepsilon_{n_k} Q_k(x-x_k).
\end{equation*}
Then from the definition of $\mathcal{H}_l$, it follows that
\begin{equation}\label{final_t_2}
|f(x) - S_{3n_l}(f, x)| \geqslant \frac{c}{2}
\varepsilon_{3n_l} m_l^{1/p}.
\end{equation}

Set $F(x) = \frac{f(x)}{C'(p, q')}$. Summarising inequality~\eqref{E_n},
estimate~\eqref{final_t_2}, and relation~\eqref{m_k_2}, we obtain that for
all points of $\mathcal{H}$, and hence for almost all $x \in \mathbb{T}$,
the conclusions~\eqref{E_n(f)_cond} and~\eqref{limsup} hold.
Theorem~\ref{theorem2} is proved.

\begin{remark}
Although the general scheme of the proof of Theorem~\ref{theorem2} follows
the work of K.\,I.~Oskolkov~\cite[Theorem~2]{Oskolkov1}, the transfer of the
result from the space of continuous functions to $L^p(\mathbb{T}),\;
1<p<\infty$, required overcoming a number of substantial difficulties related
to the specifics of $L^p$-norms. The main difficulty is the following.
In~\cite{Oskolkov1}, uniformly bounded polynomials on $\mathbb{T}$ were used,
and the behaviour of the polynomials was controlled uniformly on $\mathbb{T}$.
In our case, it was necessary to construct the polynomials $Q_k$ so that their
$L^p$-norms are uniformly bounded, while the deviations
$|Q_k(x) - S_{3n_k}(Q_k,x)|$ on the sets $A_k$ of sufficiently large measure
are of order $m_k^{1/p}$ (estimate~\eqref{estimate2}).
\end{remark}

\section{Proof of Theorem~\ref{theorem3}}

To prove the Prinsheim convergence of double Fourier series of functions
satisfying the required condition on the modulus of continuity, we first use
the scheme of the proof of Theorem~\ref{theorem1}, namely the proof of the
estimate of the rate of convergence of the one-dimensional Fourier series
in terms of the modulus of continuity. Then we use the idea of the proof
of the analogous theorem of K.\,I.~Oskolkov for continuous functions.
However, in our situation, when considering functions from $L^p$, substantial
technical difficulties arise; in particular, we use ideas from the work
of V.\,I.~Kolyada~\cite{Kolyada} on the relation between the modulus of
continuity and its one-dimensional restriction (Lemma~\ref{lemma}).

We first obtain an estimate of the rate of convergence of the Fourier series
of the function $f_t$, where $f_t := f(t, \cdot) \in L^p(\mathbb{T})$.
To this end, set $m_k := 2^{2^k}$ for $k \in \mathbb{N}$. By
condition~\eqref{cond_t_3}, there exists a non-decreasing function
$\theta: [0, +\infty) \rightarrow [0, +\infty)$ such that
$\ln{n} = o(\theta(n))$ as $n \to \infty$ and
\begin{gather}
\sum\limits_{n=1}^{\infty}
\frac{(\omega(f, 1/n)_p \theta(n))^p}{n \ln{n}} < \infty,\label{cond_t_3'} \\
\theta(m_{k+1}) \leqslant 4 \theta(m_k),\quad k \in \mathbb{N}. \label{theta_cond}
\end{gather}

Fix a natural number $n$ and consider the expression
\begin{equation}\label{M_n(f)_2}
M_n(f, x, t) :=
\max\limits_{m_2 < m \leqslant n} |f_t(x) - S_m(f_t, x)| \theta(m),\quad x \in \mathbb{T}.
\end{equation}

We apply the scheme of the proof of Theorem~\ref{theorem1}. Namely, for
$(t, x) \in \mathbb{T}^2$, denote by $m_t(x)$ the natural number in
$(m_2, n]$ at which the maximum in~\eqref{M_n(f)_2} is attained. Let $k$ be
such that $m_t(x) \in (m_k, m_{k+1}]$. Denote
$g_{k, t}(x) := f_t(x) - T_{k, t}(x)$, where $T_{k, t}$ is a trigonometric
polynomial of best approximation of order $m_k + 1$ for $f_t$ in
$L^p(\mathbb{T})$. Then
\begin{equation}\label{diff_eq}
f_t(x) - S_{m_t(x)}(f_t, x) = g_{k,t}(x) - S_{m_t(x)}(g_{k,t}, x).
\end{equation}
By Theorem~\ref{Hunt}, we have
\begin{equation}\label{Hunt_2}
|g_{k,t}(x) - S_{m_t(x)}(g_{k,t}, x)| \le 2 S^*(g_{k,t}, x)
\end{equation}
for almost all $x$. For a natural number $n$, define
$\overline{k} := \min\{k: n \leqslant m_{k+1}\}$. Then, using~\eqref{diff_eq}
and~\eqref{Hunt_2}, we obtain
\begin{multline}\label{M_n_eq}
\|M_n(f, \cdot, t)\|_p^p =
\sum\limits_{k=2}^{\overline{k}}
\int\limits_{\{x: m_t(x) \in (m_k, m_{k+1}]\}}
\left(|f_t(x) - S_{m_t(x)}(f_t, x)| \theta(m_t(x))\right)^p \, dx \\
\leqslant 2^p \sum\limits_{k=2}^{\overline{k}}
\theta^p(m_{k+1}) \|S^*(g_{k, t}, \cdot)\|_p^p
\leqslant 2^p \overline{C}_p
\sum\limits_{k=2}^{\overline{k}}
\theta^p(m_{k+1}) \|g_{k,t}\|_p^p.
\end{multline}
Since $\|g_{k,t}\|_p = E_{m_k+1}(f_t)_p$, using Jackson's inequality and inequality~\eqref{theta_cond}, from~\eqref{M_n_eq}
we derive
\begin{equation*}
\|M_n(f, \cdot, t)\|_p^p \leqslant C
\sum\limits_{k=1}^{\overline{k}-1}
\theta^p(m_k) \omega\left(f_t, \frac{1}{m_{k+1}}\right)_p^p
\end{equation*}
for some $C>0$ independent of $t$. From this and Lemma~\ref{lemma}, it follows that
\begin{equation}\label{int_M_n}
\int\limits_{\mathbb{T}}\|M_n(f, \cdot, t)\|_p^p \, dt
\leqslant C_1
\sum\limits_{k=1}^{\overline{k}-1}
\theta^p(m_k) \omega\left(f, \frac{1}{m_{k+1}}\right)_p^p
\end{equation}
for some $C_1 > 0$. Note that
\begin{multline}\label{ineq_sum}
\sum\limits_{n=m_1}^{\infty}
\frac{(\omega(f, 1/n)_p \theta(n))^p}{n \ln{n}}
= \sum\limits_{k=1}^{\infty}
\sum\limits_{n=m_k}^{m_{k+1}-1}
\frac{(\omega(f, 1/n)_p \theta(n))^p}{n \ln{n}} \\
\geqslant \sum\limits_{k=1}^{\infty}
\theta^{p}(m_k) \omega\left(f, \frac{1}{m_{k+1}} \right)_p^p
\ln{\left(\frac{\ln{m_{k+1}}}{\ln{m_k}}\right)}.
\end{multline}
By the choice of the sequence $m_k$, the expression
\begin{equation*}
\ln{\left(\frac{\ln{m_{k+1}}}{\ln{m_k}}\right)}
\end{equation*}
is bounded below in $k$. Hence, from~\eqref{ineq_sum} and
condition~\eqref{cond_t_3'}, we obtain
\begin{equation}\label{series}
\sum\limits_{k=1}^{\infty}
\theta^p(m_k) \omega\left(f, \frac{1}{m_{k+1}}\right)_p^p < \infty.
\end{equation}
From the definition~\eqref{M_n(f)_2}, inequality~\eqref{int_M_n}, property~\eqref{series}, and monotone convergence
theorem, it follows that for all natural $m > m_2$
and $(t, x) \in \mathbb{T}^2$,
\begin{equation}\label{ineq}
|f_t(x) - S_m(f_t, x)| \leqslant \frac{C(t, x)}{\theta(m)},
\end{equation}
where $C(t, x) \in L^p(\mathbb{T}^2)$ and $\ln{n} = o(\theta(n))$ as
$n \to \infty$.

Let $m_2 < n_1 \leqslant n_2$. Consider the difference
\begin{multline}\label{ineq_t_3}
f(x_1, x_2) - S_{n_1, n_2}(f, x_1, x_2)  \\
= f(x_1, x_2) - \frac{1}{\pi^2}
\int\limits_{\mathbb{T}^2}
D_{n_1}(t_1 - x_1) D_{n_2}(t_2 - x_2) f(t_1, t_2) \, dt_2 \, dt_1 \\
= I_1(x_1, x_2) + I_2(x_1, x_2),
\end{multline}
where $D_{n_i}$ are the Dirichlet kernels for $i \in \{1, 2\}$ and
\begin{gather*}
I_0(t_1, x_2) := f(t_1, x_2) -
\frac{1}{\pi} \int\limits_{\mathbb{T}}
D_{n_2}(t_2 - x_2) f(t_1, t_2) \, dt_2,\\
I_1(x_1, x_2) := \frac{1}{\pi}
\int\limits_{\mathbb{T}}
D_{n_1}(t_1 - x_1) I_0(t_1, x_2) \, dt_1, \\
I_2(x_1, x_2) := f(x_1, x_2) -
\frac{1}{\pi} \int\limits_{\mathbb{T}}
D_{n_1}(t_1 - x_1) f(t_1, x_2) \, dt_1.
\end{gather*}

We estimate the quantities $I_1(x_1, x_2)$ and $I_2(x_1, x_2)$.
To estimate $I_1(x_1, x_2)$, we first estimate $I_0(t_1, x_2)$.
By~\eqref{ineq}, we have
\begin{equation}\label{C_1}
\left| I_0(t_1, x_2) \right| \leqslant
\frac{C(t_1, x_2)}{\theta(n_2)} =: \psi(t_1, x_2),
\end{equation}
where $C(t_1, x_2) \in L^p(\mathbb{T}^2)$. From the representation of the
Dirichlet kernel (see, e.g.,~\cite[Ch.~1, \S~32]{Bari}), it follows that
\begin{equation}\label{Dir}
I_1(x_1, x_2) = \frac{1}{\pi}
\int\limits_{\mathbb{T}} \frac{\sin{n_1t_1}}{t_1}
I_0(t_1 + x_1, x_2) \, dt_1 + o(1)
\end{equation}
as $n_1 \to \infty$. Let
\begin{equation*}
\Psi_{x_1}(x, x_2) :=
\int\limits_0^{x} (\psi(x_1 +t_1,x_2) + \psi(x_1 - t_1, x_2)) \, dt_1.
\end{equation*}
Then from~\eqref{Dir} and inequality~\eqref{C_1}, we obtain
\begin{multline*}
|I_1(x_1, x_2)| \leqslant \frac{1}{\pi}
\int\limits_{-\pi}^{\pi} \left|\frac{\sin{n_1t_1}}{t_1} \right|
\psi(x_1+t_1, x_2) \, dt_1 + o(1)\\
\leqslant n_1 \Psi_{x_1}\left(\frac{1}{n_1}, x_2 \right)
+\int\limits_{1/n_1}^{\pi}
\frac{\psi(x_1 + t_1,x_2) + \psi(x_1-t_1, x_2)}{t_1}\, dt_1 + o(1) \\
= \frac{\Psi_{x_1}(\pi, x_2)}{\pi}
+ \int\limits_{1/n_1}^{\pi}
t_1^{-2} \Psi_{x_1}(t_1, x_2) \, dt_1 + o(1) \\
\leqslant \mathcal{M}\psi_{x_2}(x_1)
\left(1 + \int\limits_{1/n_1}^{\pi} t_1^{-1} \, dt \right) + o(1)
\leqslant B \mathcal{M}C_{x_2}(x_1) \frac{\ln{n_1}}{\theta(n_2)} + o(1),
\end{multline*}
where $B := \ln{\pi} + 2$, $\mathcal{M}\psi_{x_2}(x_1)$ for fixed $x_2$
denotes the application of the Hardy--Littlewood maximal function to
$\psi(\cdot, x_2)$, and $\mathcal{M}C_{x_2}(x_1)$ is understood analogously.
Using Theorem~\ref{maximal}, we estimate
\begin{equation*}
\int\limits_{\mathbb{T}^2} |\mathcal{M}C_{x_2}(x_1)|^p \, dx_1 \, dx_2
= \int\limits_{\mathbb{T}} \| \mathcal{M}C_{x_2} \|_p^p \, dx_2
\leqslant A_p \int\limits_{\mathbb{T}^2} |C(x_1, x_2)|^p \, dx_2 \, dx_1
< +\infty.
\end{equation*}
Hence $\mathcal{M}C_{x_2}(x_1)$ is finite for almost all
$(x_1, x_2) \in \mathbb{T}^2$.

Analogously to estimate~\eqref{C_1}, using Lemma~\ref{lemma}, the expression
$|I_2(x_1, x_2)|$ is estimated. Thus, for $n_2 \geqslant n_1$,
expression~\eqref{ineq_t_3} is estimated as follows:
\begin{equation*}
\left| f(x_1, x_2) - S_{n_1, n_2}(f, x_1, x_2) \right|
\leqslant
B \mathcal{M}C_{x_2}(x_1) \frac{\ln{n_2}}{\theta(n_2)}
+  C_1(x_1, x_2) \frac{1}{\theta{(n_1})} + o(1),
\end{equation*}
as $n_1 \to \infty$, where $C_1(x_1, x_2)$ denotes a non-negative function finite
almost everywhere and $\ln{n} = o(\theta(n))$ as $n \to \infty$.
From the last inequality it follows that almost everywhere
\begin{equation*}
\lim\limits_{n_1 \to \infty}
\sup\limits_{n_2 \geqslant n_1}
S_{n_1, n_2}(f, x_1, x_2) = f(x_1, x_2).
\end{equation*}
Carrying out analogous arguments in the case $n_2 < n_1$, we obtain that
almost everywhere
\begin{equation*}
\lim\limits_{n_2 \to \infty}
\sup\limits_{n_1 > n_2}
S_{n_1, n_2}(f, x_1, x_2) = f(x_1, x_2).
\end{equation*}
Combining the two last limiting relations, we conclude that
\begin{equation*}
\lim_{\min{\{n_1, n_2\}} \to \infty}
S_{n_1, n_2}(f, x_1, x_2) = f(x_1, x_2)
\end{equation*}
for almost all $(x_1, x_2) \in \mathbb{T}^2$. Theorem~\ref{theorem3} is proved.

\medskip
The author expresses his gratitude to N.\,Yu.~Antonov for setting the problem
and for discussions of the work. The author is also grateful to the referee
for valuable comments that improved the text and the results of the paper.



\begin{thebibliography}{99}

\bibitem{Lebesgue}
H.~Lebesgue,
``Sur les int\'egrales singuli\`eres,''
\emph{Ann. Toulouse} \textbf{1} (1909), 25--117.

\bibitem{Bari}
N.\,K.~Bari,
\emph{Trigonometric Series} (in Russian),
Fizmatlit, Moscow, 1961.

\bibitem{Oskolkov1}
K.\,I.~Oskolkov,
``Estimate of the rate of approximation of a continuous function
and its conjugate by Fourier sums on a set of full measure,''
\emph{Izv. Akad. Nauk SSSR Ser. Mat.} \textbf{38} (1974), no.~6, 1393--1407
(in Russian); English transl. in \emph{Math. USSR-Izv.} \textbf{8} (1974), no.~6, 1377--1391.

\bibitem{Oskolkov2}
K.\,I.~Oskolkov,
``Approximative properties of summable functions on sets of full measure,''
\emph{Mat. Sb.} \textbf{103} (1977), no.~4, 563--589 (in Russian);
English transl. in \emph{Math. USSR-Sb.} \textbf{32} (1977), no.~4, 489--514.

\bibitem{Oskolkov3}
K.\,I.~Oskolkov,
``Subsequences of Fourier sums of integrable functions,''
in \emph{Modern Problems of Mathematics. Mathematical Analysis, Algebra, Topology}
(collection of articles dedicated to Academician L.\,S.~Pontryagin on his 75th birthday),
\emph{Trudy Mat. Inst. Steklov} \textbf{167} (1985), 239--260 (in Russian);
English transl. in \emph{Proc. Steklov Inst. Math.} \textbf{167} (1986), 259--282.

\bibitem{Quade}
E.\,S.~Quade,
``Trigonometric approximation in the mean,''
\emph{Duke Math. J.} \textbf{3} (1937), no.~3, 529--543.

\bibitem{Ulyanov}
P.\,L.~Ul'yanov,
``On absolute and uniform convergence of Fourier series,''
\emph{Mat. Sb.} \textbf{72 (114)} (1967), no.~2, 193--225 (in Russian);
English transl. in \emph{Math. USSR-Sb.} \textbf{1} (1967), no.~2, 173--203.

\bibitem{KonStech}
A.\,A.~Konyushkov,
``Best approximations by trigonometric polynomials and Fourier coefficients,''
\emph{Mat. Sb.} \textbf{44 (86)} (1958), no.~1, 53--84 (in Russian);
English transl. in \emph{Math. USSR-Sb.} \textbf{44} (1958), no.~1, 53--84.

\bibitem{Fefferman}
C.~Fefferman,
``On the divergence of multiple Fourier series,''
\emph{Bull. Amer. Math. Soc.} \textbf{77} (1971), no.~2, 191--195.

\bibitem{BahNik}
M.~Bakhbukh and E.\,M.~Nikishin,
``On the convergence of double Fourier series of continuous functions,''
\emph{Sibirsk. Mat. Zh.} \textbf{14} (1973), no.~6, 1189--1199 (in Russian);
English transl. in \emph{Siberian Math. J.} \textbf{14} (1973), no.~6, 824--832.

\bibitem{Bak97}
A.\,N.~Bakhvalov,
``On the divergence everywhere of Fourier series of continuous functions
of several variables,''
\emph{Mat. Sb.} \textbf{188} (1997), no.~8, 45--62 (in Russian);
English transl. in \emph{Sb. Math.} \textbf{188} (1997), no.~8, 1171--1188.

\bibitem{Zhizhiashvili}
L.\,V.~Zhizhiashvili,
``On the convergence of multiple trigonometric Fourier series,''
\emph{Soobshch. Akad. Nauk Gruzin. SSR} \textbf{80} (1975), no.~1, 17--19 (in Russian).

\bibitem{Stokolos}
A.\,M.~Stokolos,
``On the strong differentiation of integrals of functions from H\"older classes,''
\emph{Mat. Zametki} \textbf{55} (1994), no.~1, 84--104 (in Russian);
English transl. in \emph{Math. Notes} \textbf{55} (1994), no.~1, 57--71.

\bibitem{Dya92}
M.\,I.~D'yachenko,
``Some problems in the theory of multiple trigonometric series,''
\emph{Uspekhi Mat. Nauk} \textbf{47} (1992), no.~5, 97--162 (in Russian);
English transl. in \emph{Russian Math. Surveys} \textbf{47} (1992), no.~5, 103--171.

\bibitem{Goginava2026}
U.~Goginava,
``An Oskolkov--Zhizhiashvili Criterion for Rectangular Fourier Sums,''
\emph{arXiv:2606.12920} [math.CA] (2026),
\href{https://arxiv.org/abs/2606.12920}{https://arxiv.org/abs/2606.12920}.

\bibitem{Goginava2026b}
U.~Goginava,
``A Critical-Scale Extension of Zhizhiashvili's Theorem for Rectangular Fourier Series,''
\emph{arXiv:2606.21637} [math.CA] (2026),
\href{https://arxiv.org/abs/2606.21637}{https://arxiv.org/abs/2606.21637}.

\bibitem{Carleson}
L.~Carleson,
``On convergence and growth of partial sums of Fourier series,''
\emph{Acta Math.} \textbf{116} (1966), no.~1, 135--157.

\bibitem{Hunt}
R.~Hunt,
``On the convergence of Fourier series,''
in \emph{Orthogonal Expansions and their Continuous Analogues},
SIU Press, Carbondale, 1968, pp.~235--255.

\bibitem{Zigmund2}
A.~Zygmund,
\emph{Trigonometric Series}, Vol.~2,
Cambridge Univ. Press, Cambridge, 1959;
Russian transl.: Mir, Moscow, 1965.

\bibitem{HL1}
G.\,H.~Hardy and J.\,S.~Littlewood,
``A maximal theorem with function-theoretic applications,''
\emph{Acta Math.} \textbf{54} (1930), 81--116.

\bibitem{Zigmund1}
A.~Zygmund,
\emph{Trigonometric Series}, Vol.~1,
Cambridge Univ. Press, Cambridge, 1959;
Russian transl.: Mir, Moscow, 1965.

\bibitem{HL2}
G.\,H.~Hardy and J.\,S.~Littlewood,
``Notes on the theory of series. XIII: Some new properties of Fourier constants,''
\emph{J. Lond. Math. Soc.} \textbf{6} (1931), 3--9.

\bibitem{PP}
P.\,P.~Petrushev and V.\,A.~Popov,
\emph{Rational Approximation of Real Functions},
Cambridge Univ. Press, Cambridge, 1987.

\bibitem{Kolyada}
V.\,I.~Kolyada,
``$L^p$-moduli of continuity of sections of functions,''
\emph{Forum Mathematicum} \textbf{22} (2010), no.~1, 53--73.

\end{thebibliography}
\end{document}